\documentclass[12pt]{article}

\usepackage{amssymb}
\usepackage{amsthm}\usepackage{subfigure}
\usepackage{amssymb}
\usepackage{amsmath}
\usepackage{color}
\usepackage{graphicx,epstopdf}
\usepackage{float}

\RequirePackage[OT1]{fontenc}
\numberwithin{equation}{section}
\theoremstyle{plain}
\newtheorem{theorem}{Theorem}[section]
\newtheorem{lemma}{Lemma}[section]

\newtheorem{remark}{Remark}[section]
\def\ud{\, \mathrm{d}}
\def\E{\mathbb E}
\def\F{\mathcal F}

\begin{document}



\title{Pricing the zero-coupon bond of the extended Cox-Ingersoll-Ross model using Malliavin calculus}

\date{}

\author{Hongyi Chen\thanks{Fordham University. Email:\ hchen102@fordham.edu} \ Sixian Jin\thanks{Corresponding author. Fordham University. Email:\ sjin27@fordham.edu} \ and Di Kang\thanks{McMaster University. Email:\ kangd18@mcmaster.ca}}


\maketitle



\begin{abstract}
In this paper, we price the zero-coupon bond of the extended Cox-Ingersoll-Ross model by a Dyson type formula established in one of the authors' paper Jin, Peng and Schelllhorn (2016) using Malliavin calculus. This formula provides a fast convergent series to represent the bond price, and it depends on the given drift and volatility of the interest rate process but not the instantaneous forward rate used in Maghsoodi (1996). This expression can be also regarded as a new solution to a class of Reccati equations with time-dependent coefficients.\\
\\
\textbf{Keywords:} Cox-Ingersoll-Ross model, Zero-coupon bond, Bond pricing formula, Malliavin calculus, Dyson type formula.\\
\textbf{MSC(2010)}: 60H07, 60G15, 91G20, 91G30
\end{abstract}




\section{Introduction}

The Cox-Ingersoll-Ross (CIR) process is a short rate model used in the
pricing of interest rate derivatives $r$ and is given by the following It\^o-type
stochastic differential equation (SDE),
\begin{equation}
\label{First}
\mathrm{d}r_s=a\left(r_s+b\right)\mathrm{d}
s+\sigma\sqrt{r_s}\mathrm{d}W_s,
\end{equation}
where $W$ is a standard Brownian motion and $a, b,$ and $\sigma$ are positive constants. This model was first introduced in \cite{CIR}, and popularized due to the use of volatility modeling by Heston in \cite{Hes93}. Nowadays, due to its desirable
properties, such as non-negativity, mean-reversion and analytical tractability, it plays a key role in the field of option pricing, for instance when modeling squared volatilities in the Heston model \cite{Hes93}.  For practitioners, however, the main shortcoming of the constant parameters version of the model \eqref{First} is that it cannot reproduce the original term structure of interest rates. This fact was highlighted in several papers such as \cite{HW90,KS08,Yang06}.

In order to adapt the CIR model to be more consistent with the current term
structure of interest rates, Hull and White \cite{HW90} extended the CIR model to the one with time-dependent coefficients as 
\begin{equation}
\label{Second}
\mathrm{d}r_s=k(s)(\theta(s)-r_s)\mathrm{d}
s+\sigma(s)\sqrt{r_s}\mathrm{d}W_s, 
\end{equation}
where $k, \theta$ and $\sigma$ are nonrandom one-variable functions. It is simply
defined by the diffusion process, and a unique strong solution to this process exists without lying on a very strong assumption, see \cite{HW90}. This model is called the extended  CIR (ECIR) model, which is frequently used to explain the behavior of term structure of interest rates, especially prices of zero-coupon bonds at any time $t$ for maturity date $T\ge t$, which is $$P(t,T):=\E\left[\exp{\left(-\int_t^Tr_s\ud s\right)}\bigg|\F_t\right].$$

However, the general bond price is very hard to calculate analytically. In \cite{Mag96}, Maghsoodi presented the interest rate in ECIR model as a sum of squares of Ornstein-Uhlenbeck processes when the dimension $d(t):=4k(t)\theta(t)/\sigma^2(t)$ is a fixed positive integer, thus it follows the generalized chi-square distribution.  A closed form for the interest rate when $d$ is a positive real value is also given in \cite{Mag96}. Then, the term structure dynamics of the ECIR model with a closed form of the discount bond price is given as $$P(t,T)=A(t,T)e^{-B(t,T)r_t},$$
where $B(t,T)$ is a solution to a time-dependent Riccati equation, and $A(t,T)$ can be obtained from $B(t,T)$. In \cite{Mag96}, the expression of $B(t,T)$ depends on the initial boundary condition $B(0,s)$ with $s\in[t,T]$, which is still unknown. Remark 3.1 of \cite{Mag96} gives a way to estimate $B(0,s)$ using historical data with the instantaneous forward rate being supposed to be known, see \cite{HW90} also.  

In our paper, we apply a Dyson type formula established in \cite{JPS16} via Malliavin calculus to calculate $P(t,T)$ to an analytical series for positive integer dimension. The main advantage of our expression is that we do not need the prior information of the initial boundary condition $B(0,s)$. This expression can also be understood as a new solution to a time-dependent Riccati equation. In practice, comparing to the Monte-Carlo simulation which requires a lot of computation time (see \cite{HM05} for example), our expression is a rapidly convergent series and the first few terms involve only simple multiple integrals of nonrandom functions which are time-saving in numerical or manual calculations. 

We organize this paper as the following. We first introduce some basic concepts and rules of Malliavin calculus and the Dyson type formula in section 2. In section 3, we state our main results of the analytical expressions for the bond price. In section 4, we provide some numerical examples. Finally, Appendix collects the proofs of technical lemmas.

\section{Notations and Preliminaries}
In this section, we briefly review some basic Malliavin calculus and
introduce the Dyson type formula for a Brownian martingale established in \cite{JPS16}. We denote by $%
(\Omega ,\mathcal{F},\mathcal{\{F}_{t}\},\mathbb{P})$ the complete filtered
probability space, where the filtration $\mathcal{\{F}_{t}\}$ satisfies the
usual conditions, i.e., it is the usual augmentation of the filtration
generated by Brownian motion $W$ on $\mathbb{R}$. Unless stated otherwise all equations involving random variables are to be regarded to
hold $\mathbb{P}$-almost surely.

Following \cite{Oks}, we say that a real function $g:[0,T]\rightarrow
\mathbb{R}^{n}$ is symmetric if:
\begin{equation*}
g(x_{\sigma (1)},\ldots,x_{\sigma(n)})=g(x_{1},\ldots,x_{n}),
\end{equation*}
for all permutations $\sigma $ on $\{1,2,..,n\}$. If in addition, $g\in
L^2([0,T]^n)$, i.e.,
\begin{equation*}
||g||_{L^{2}([0,T]^{n})}^{2}=\int_{0}^{T}\ldots
\int_{0}^{T}g^{2}(x_{1},\ldots,x_{n})\, \mathrm{d} x_{1}\ldots\, \mathrm{d}
x_{n}<+\infty,
\end{equation*}
then we say $g$ belongs to $\hat{L}^{2}([0,T]^{n})$, the space of symmetric
square-integrable functions on $[0,T]^{n}$. Denote by $L^2(\Omega)$ the
space of random variables with finite variances, i.e., 
\begin{equation*}
\|F\|_{L^2(\Omega)}:=\sqrt{\E[F^2]}<+\infty.
\end{equation*}
The Wiener chaos expansion of $F\in L^2(\Omega)$, is defined by
\begin{equation}
F=\sum_{m=0}^{\infty }I_{m}(f_{m})\text{ \ \ in }L^{2}(\Omega ),
\label{WienerChaos}
\end{equation}
where $\{f_{m}\}_{m=0}^{\infty }$ is a uniquely determined sequence of
nonrandom functions ($m$-dimensional kernels) in $\hat{L}^{2}([0,T]^{n})$, and the operator $I_m:\hat{L}^{2}([0,T]^{n})\rightarrow L^2(\Omega)$ is
defined as
\begin{equation*}
\left\{
\begin{array}{lll}
I_{m}(f_{m}) & := & \int_{0}^{T}\int_{0}^{t_{m}}...
\int_{0}^{t_{2}}f_m(t_{1},...,t_{m})\, \mathrm{d} W_{t_{1}}\, \mathrm{d}
W_{t_{2}}...\, \mathrm{d} W_{t_{m}},~\text{ if }m>0; \\
I_{0}(f_{0}) & := & f_{0}.
\end{array}
\right.
\end{equation*}

Following Lemma 4.16 in \cite{Nua}, the Malliavin derivative $D_{t}F$ of $F$, when it exists, satisfies
\begin{equation*}
D_{t}F=\sum_{m=0}^{\infty }mI_{m-1}(f_{m}(\cdot,t))\text{ \ \ in }
L^{2}(\Omega).
\end{equation*}
We denote the Malliavin derivative of order $l$ of $F$ at time $t$ by $
D_{t}^{l}F$, as an abbreviation of $\underbrace{D_t\ldots D_t}_{l~
\mbox{times}}F$. We call $\mathbb{D}^{\infty }([0,T])$ the set of random
variables which are infinitely Malliavin differentiable and $\mathcal{F}_{T}$
-measurable. A random variable is said to be infinitely Malliavin
differentiable if for any positive integer $n$,
\begin{equation}  \label{dinfty}
\E\left[ \left( \sup_{s_{1},\ldots ,s_{n}\in \lbrack 0,T]}\big|%
D_{s_{n}}\ldots D_{s_{1}}F\big|\right) ^{2}\right] <+\infty .
\end{equation}
In particular, we denote by $\mathbb{D}^N([0,T])$ the space of all random
variables $F$ which satisfy (\ref{dinfty}) for all $n\leq N$. The Malliavin derivative satisfies the usual product and chain rules:
\begin{enumerate}
\item Product rule: for any $t\ge 0$,
$$D_t(FG)=FD_tG+GD_tF.$$
\item Chain rule: if $f(x_1,...,x_n)$ is an $n$-variable differentiable function, then
$$
D_{s}f(W_{t_{1}},...,W_{t_{n}})=\sum_{i=1}^{n}\frac{\partial f}{\partial
x_{i}}(W_{t_{1}},...,W_{t_{n}})1_{[0,t_i]}(s).
$$
\end{enumerate}

Given $\omega \in \Omega $, we define the freezing path operator $\omega ^{t}$
as
\begin{equation*}
( \omega ^{t}\circ F)(\omega ):=F(\omega ^{t}(\omega )),
\end{equation*}
where $\omega ^{t}(\omega )$ represents a ``frozen path" - a
particular path where the corresponding Brownian motion becomes a constant
after time $t$. For example,
\begin{equation}  \label{defineomegat}
\omega^t\circ W_s(\omega)=W_s(\omega ^{t}(\omega ))=\left\{%
\begin{array}{l}
W_s(\omega )~\mbox{if}~ s\leq t; \\
W_t(\omega )~\mbox{if}~t\leq s\leq T.%
\end{array}%
\right.
\end{equation}
In the remainder of text, for convenience, we only write $\omega ^{t}$ instead of $\omega ^{t}(\omega )$, and $F$ instead of $F(\omega)$. We give in the following remark some examples for the freezing path operator, which can also be viewed as a constructive definition of this operator. For more details, we refer to Section 2.1 of \cite{JPS16}.
\begin{remark}
Let $t\le T$:
\begin{enumerate}
\item Suppose $%
F=p(W_{s_{1}},\ldots,W_{s_{n}})$, where $p$ is a polynomial. Then $\omega ^{t}\circ F=p(W_{s_{1}\wedge
t},\ldots,W_{s_{n}\wedge t})$. 
\item The following equations hold:%
\begin{eqnarray*}
&&\omega
^{t}\circ\left( \int_{0}^{T}f(s)\, \mathrm{d} W_s\right) =\int_{0}^{t}f(s)\, \mathrm{d} W_s; \\
&&\omega
^{t}\circ\left( \int_{0}^{T}W_s\, \mathrm{d} s\right) =\int_{0}^{t}W_s\, \mathrm{d} s+W_t(T-t).
\end{eqnarray*}%
\item Acting a freezing path operator on an It\^o integral is more complicated, for example,\begin{equation*}
\omega
^{t}\circ\left( \int_{0}^{T}W_s\, \mathrm{d} W_s\right)=\omega ^{t}\circ\left( \frac{%
W_T^{2}-T}{2}\right)=\frac{W_t^{2}-T}{2}.
\end{equation*} 
For a general case $\omega^t\circ I_m(f_m)$, we refer to Proposition 2 in \cite{JS17}. 

\item It is worth emphasizing the following fact to readers, that is
\begin{equation*}
\omega^t\circ (D_sF)\neq D_s\left(\omega^t\circ F\right).
\end{equation*}
For instance, if $t< s\leq T$, then $\omega^t\circ 
(D_{s}W_T^{2})=\omega^t\circ (2W_T)=2W_t$, but $D_s(\omega^t\circ W_T^2)=D_sW_t^2=0$. 
\end{enumerate}
\end{remark}

Now we state the Dyson type formula, which is Theorem 2.3 in \cite{JPS16}.
\begin{theorem}
\label{EFthm} 
Suppose that a random variable $F$ is $\mathcal{F}_T$-measurable and satisfies the following condition
\begin{equation}
\label{conv}
\frac{(T-t)^{2n}}{(2^{n}n!)^{2}}\E\left[ \left( \sup_{u_{1},...,u_{n}\in
\lbrack t,T]}\left\vert \omega^t \circ (D_{u_{n}}^{2}\ldots D_{u_{1}}^{2}F)\right\vert \right) ^{2}\right] \xrightarrow[]{n\rightarrow \infty}0
\end{equation}%
for a fixed $t\in \lbrack 0,T]$. Then in $L^2(\Omega)$,
\begin{equation}
 \label{EF}
\E[F|\mathcal{F}_{t}]=\sum_{n=0}^{+\infty }\frac{1}{2^{n}n!}%
\int_{[t,T]^{n}}\omega ^{t}\circ(D_{s_{n}}^{2}\ldots D_{s_{1}}^{2}F)\,\mathrm{d}%
s_{n}\ldots \,\mathrm{d}s_{1}. 
\end{equation}
\end{theorem}

\section{Zero-Coupon Bond Price of the Extended CIR Model}

In this section, we use Theorem \ref{EFthm} to price the zero-coupon bond of the extended CIR model defined by
\begin{equation}
\label{CIR}
\, \mathrm{d} r_s=2k(s)\left(\theta(s)-r_s\right) \, \mathrm{d}
s+2\sigma (s)\sqrt{r_s}\, \mathrm{d} W_s,
\end{equation}%
where the initial value is a constant $r_0$, $k(s)$ and $\sigma (s)$ are nonrandom functions which satisfy $d=2k(t)\theta(t)/\sigma^2(t)$ being a fixed positive integer. Note that this matches \eqref{Second} with constant multiplying. 

By Theorem 2.1 of \cite{Mag96}, the interest
rate $\{r_s\}_{s\ge0}$ can be represented as
\begin{equation}  \label{rel}
\{r_s\}_{s\ge0}=\left\{\sum_{i=1}^{d}(X^{(i)}_s)^2\right\}_{s\ge0},
\end{equation}%
where $\{X^{(i)}\}_i$ are i.i.d. Ornstein-Uhlenbeck processes defined by
\begin{equation}
\label{OUprocess}
\, \mathrm{d} X^{(i)}_s=-k(s)X^{(i)}_s\, \mathrm{d} s+\sigma (s)\, \mathrm{d}
W^{(i)}_s,
\end{equation}%
with $X^{(i)}_{0}=x_0=\sqrt{\frac{r_{0}}{d}}$ and $W^{(i)}\in(\Omega,\{\mathcal F_t^{(i)}\}_{t\ge0},\mathbb{P})$ being independent standard Brownian motions for all $i=1,\ldots,d$. Thus $X^{(i)}$ has an explicit expression 
\begin{equation}\label{Xi}
X^{(i)}_t=x_0e^{-\int_0^tk(u)\ud u}+\int_0^te^{-\int_u^tk(r)\ud r}\sigma(u)\ud W^{(i)}_u.
\end{equation}
Let
\begin{equation*}
F=\exp \Big( -\int_{t}^{T}r(s)\, \mathrm{d} s\Big) .
\end{equation*}%
Our goal is to find the bond price $P(t,T)=\E[F|\mathcal{F}_{t}]$ for any fixed $t\in[0,T]$. Since $F$ can be
written as $F=\prod_{i=1}^{d}F^{(i)}$ where $F^{(i)}:=\exp (
-\int_{t}^{T}(X^{(i)}_s)^{2}\, \mathrm{d} s) $, we can decompose $P(t,T)$ into a product of conditional expectations due to the independence of $X^{(i)}$, which is
\begin{eqnarray}
P(t,T)&=&\E[F|\mathcal{F}_{t}]=\prod_{i=1}^{d}\E[F^{(i)}|\mathcal{F}_{t}^{(i)}]\notag\\
&=&\prod_{i=1}^{d}\E\left[\exp{\left(-\int_{t}^{T}(X^{(i)}_s)^2\ud s \right)}\bigg|\mathcal{F}_{t}^{(i)}\right],
\label{decomCIR}
\end{eqnarray}
where $\F^{(i)}$ is the natural filtration of Brownian motion $W^{(i)}$. 

Now, we apply Theorem \ref{EFthm} to calculate $P(t,T)$ by an analytical series which only depends on drift $k$ and volatility $\sigma$. To guarantee the convergence of the series, we first present the following lemma.
\begin{lemma}
\label{CONV}
If $F=\exp \Big( -\int_{t}^{T}r_s\, \mathrm{d} s\Big) $ where $r$ is the solution to SDE \eqref{CIR}, for any fixed $t$ and $T$, the condition \eqref{conv} is satisfied.
\end{lemma}

Applying the series \eqref{EF} on \eqref{decomCIR}, we obtain
\begin{eqnarray}
\label{EFseries}
\E[F|\F_t]=\prod_{i=1}^{d}\sum_{n=0}^{+\infty }\frac{1}{2^{n}n!}%
\int_{[t,T]^{n}}\omega ^{t}\circ(D_{s_{n}}^{2}\ldots D_{s_{1}}^{2}F^{(i)})\,\mathrm{d}%
s_{n}\ldots \,\mathrm{d}s_{1}. 
\end{eqnarray}
It is worth noticing that we apply \eqref{EF} on each of $\E[F^{(i)}|\F^{(i)}]$ separately, thus the freezing path operator $\omega$ is to freeze the sample paths of each Brownian motion $W^{(i)}$ to a constant after time $t$. To save space, we write $\omega$ without distinguishing them.

To calculate \eqref{EFseries}, the key part is the structure of the iterated Malliavin derivatives acted by the freezing path operator: $\omega ^{t}\circ(D_{s_{n}}^{2}\ldots D_{s_{1}}^{2}F^{(i)})$. We intensively discuss a special case starting from the representation \eqref{rel} instead of the solution of SDE \eqref{CIR}, in which we assume $k(t)\equiv 0$. The case with time-dependent $k=k(t)$ follows similarly. 

For detailed calculation, we prepare some notations for simplicity. Fixed $t$ and $T$, denote
$$F=e^Y\ \ \mbox{with}\ \ Y:=-\int_t^TX_s^2\ud s\ \ \mbox{and}\ \ X_s=x_0+\int_0^s\sigma(u)\ud W_u.$$
We get rid of index $i$ because $\{X^{(i)}\}_i$  are i.i.d.
To calculate $D_{s_{n}}^{2}\dots D_{s_{1}}^{2}F$, the main technical problem arises from the time variables contained in Malliavin derivatives. We list two simple results and some intuitive observations:
\begin{eqnarray*}
D_{s_{1}}^{2}F&=&F\left(\left(D_{s_{1}}Y\right)^{2}+D_{s_{1}}^2Y\right); \\
D_{s_{1}}^{2}D_{s_{2}}^{2}F&=&F\big(2\left(D_{s_{1}s_2}Y\right)^{2}
+\left(D_{s_{1}}Y\right)^{2}\left(D_{s_{2}}Y\right)^{2}+D_{s_{1}}^2Y\left(D_{s_{2}}Y\right)^{2} \\
&&+ D_{s_{2}}^2Y\left(D_{s_{1}}Y\right)^{2}+D_{s_{1}}^2YD_{s_{2}}^2Y+
4D_{s_{1}}YD_{s_{2}}YD_{s_{1}s_2}Y\big).
\end{eqnarray*}
Recall that we write $D_sD_s$ by $D_s^2$. To save space, we write $D_{s_{1}}D_{s_{2}}$ by $D_{s_{1}s_2}$.

\textbf{Observations:}
\begin{enumerate}
\item $D_{s_{n}}^{2}\ldots D_{s_{1}}^{2}F$ equals to
$F$ multiplied by an expression in terms of $Y$;

\item In each term, each of $D_{s_{1}},\ldots,D_{s_{n}}$ appears twice;

\item Taking Malliavin derivatives on $Y$ more than twice is $0$. For example, $D_{s_3}D_{s_2}D_{s_1}Y=0$.
\end{enumerate}
The structure of $D_{s_{n}}^{2}\ldots D_{s_{1}}^{2}F$
can be described by the following lemma.
\begin{lemma}
\label{Lemma2}
$D_{s_{n}}^{2}\ldots D_{s_{1}}^{2}F=F\sum_{\alpha\in{\cal B}_{n}}K_{\alpha}$
where $K_{\alpha}$ is the multiplication of terms that have forms shown in the third column of Table 1, subjected to a certain set partition $\alpha$. ${\cal B}_{n}$ denotes all partitions of set $\{1,1,2,2,...,n,n\}$
satisfying the following rules:

In each subset, there is either only one element, or all elements
are in pairs, or all elements are in pairs except two elements (see the second column in Table 1). For a certain $\alpha$, the correspondence between the subset in partition $\alpha$ and terms in $K_\alpha$ is shown in Table 1. 
\begin{table}
\begin{centering}
\scalebox{0.96}{\begin{tabular}{|c|c|c|}
\hline 
 & Subset & Derivative of $X$\tabularnewline
\hline 
Case 1 & $\{i\}$ & $D_{s_{i}}Y$\tabularnewline
\hline 
Case 2 & $\{i,i\}$ & $D_{s_{i}}^2Y$\tabularnewline
\hline 
Case 3 & $\{i_{1}.i_{1},i_{2},i_{2}\}$ & $2\left(D_{s_{i_{1}}s_{i_{2}}}Y\right)^{2}$\tabularnewline
\hline 
Case 4 & $\{i_{1},i_{1},i_{2},i_{2},...,i_{k},i_{k}\},(k\geq3)$ & $2^{k}D_{s_{i_{1}}s_{i_{2}}}YD_{s_{i_{2}}s_{i_{3}}}Y\ldots D_{s_{i_{k}}s_{i_{1}}}Y$\tabularnewline
\hline 
Case 5 & $\{i_{1},i_{2},i_{2},i_{3},i_{3},...,i_{k-1},i_{k-1},i_{k}\},(k\geq2)$ & $2^{k}D_{s_{i_{1}}s_{i_{2}}}YD_{s_{i_{2}}s_{i_{3}}}Y\ldots D_{s_{i_{k-1}}s_{i_{k}}}Y$\tabularnewline
\hline 
\end{tabular}}
\par\end{centering}
\protect\caption{Relations between subset of $\{1,1,2,2,\ldots,n,n\}$ and derivative
of $X$.}
\end{table}
\end{lemma}

Next, we establish the  analytical structure of iterated Malliavin derivatives acted on by the freezing path operator $\omega^t$, that is $\omega^t\circ (D_{s_{n}}^{2}\ldots D_{s_1}^{2}F)$.

\begin{lemma}
\label{lemmaGnm} For all non-negative integer $m$, we have
\begin{equation}
\label{Gnmdef}
\omega ^{t}\circ \left( D_{s_{1}}^{2}\ldots D_{s_{n}}^{2}F\right) :=(\omega
^{t}\circ F)\sum_{m=0}^{n}G_{n}^{m}(s_{1},...,s_{n})\left(\prod_{i=1}^n\sigma(s_i)^2\right)X_t^{2m}
\end{equation}%
where $G_{n}^{m}(s_{1},...,s_{n})$ is an $n$-variable polynomial
which can be computed by the following recurrence formulas:

1. $G_0^0=1$ and $G_{n}^{m}=0$ when $m>n$.
 
2. For $n>0$ and $m=0$,
\begin{eqnarray}
\label{recurGn0}
G_{n}^{0}&=&-2(T-s_{n})G_{n-1}^{0}(\hat s_n)\nonumber \\
&&+\sum_{k=1}^{n-1}\sum_{\{i_{1},i_{2}...i_{k}\}\subset \{1,...,n-1\}}(-1)^{k+1}2^{2k+1}(T-s_{n}\vee
s_{i_{1}})\ldots(T-s_{i_{k}}\vee s_{n})\nonumber\\
&&\ \ \times G_{n-k-1}^{0}(\hat{s}_{i_{1}},...,\hat{s}%
_{i_{k}}, \hat s_n).
\end{eqnarray}

3. For $n\ge m\ge 1$,
\begin{eqnarray}
\label{recurGnm}
G_{n}^{m}&=&\frac{1}{m}%
\sum_{i=1}^{n}4(T-s_{i})^{2}G_{n-1}^{m-1}(\hat{s}_{i})  \notag \\
&&-\frac{1}{m}\sum_{\substack{ \{i,j\}\subset \{1,...,n\}}}2^5(T-s_i)(T-s_i\vee s_j)(T-s_j)G_{n-2}^{m-1}(\hat s_i,\hat s_j)\notag\\
&&+\frac{1}{m}\sum_{k=1}^{n-2}\sum_{\substack{ \{i,j\}\subset \{1,...,n\}
\\ \{i_{1},...,i_{k}\}\subset \{1,...,n\}\backslash \{i,j\}}}%
(-1)^{k+1}2^{2k+5}\nonumber\\
&&\ \ \times (T-s_{i})(T-s_{i}\vee s_{i_{1}})\ldots(T-s_{i_{k}}\vee
s_{j})(T-s_{j}) \nonumber \\
&&\ \ \times G_{n-k-2}^{m-1}(\hat s_i, \hat s_j, \hat{s}_{i_{1}},...,\hat{s}_{i_{k}}).
\end{eqnarray}
$G(\hat s_i)$ denotes the function $G$ containing all variables $s_1,s_2,...,s_n$ except $s_i$.
\end{lemma}

Proofs of Lemma \ref{CONV}-\ref{lemmaGnm} are collected in Appendix.

Lemma \ref{lemmaGnm} gives us a way to analytically calculate the bond price $P(t,T)$.
Note that when $k=0$, $$\omega^t\circ F^{(i)}=\omega^t\circ\exp{\left(-\int_t^T(X_s^{(i)})^2\ud s\right)}=e^{-(T-t)(X_t^{(i)})^2}.$$
Combining \eqref{EFseries} and \eqref{Gnmdef}, we have
\begin{eqnarray}
&&P(t,T)\notag\\
&=&\prod_{i=1}^{d}\sum_{n=0}^{\infty }\frac{1}{2^{n}n!}%
\int_{[t,T]^{n}}\omega ^{t}\circ(D_{s_{n}}^{2}\ldots D_{s_{1}}^{2}F^{(i)})\,\mathrm{d}%
s_{n}\ldots \mathrm{d}s_{1}\notag\\
&=&e^{-(T-t)r_t}\prod_{i=1}^{d}\sum_{n=0}^{\infty }\sum_{m=0}^{n}%
\int_{[t,T]^{n}}\!\!\frac{1}{2^{n}n!}G_{n}^{m}(s_{1},...,s_{n})\prod_{j=1}^n\sigma(s_j)^2(X_t^{(i)})^{2m}\ud%
s_{n}...\ud s_{1}\notag\\
&=&e^{-(T-t)r_t}\prod_{i=1}^{d}\sum_{m=0}^{\infty }A_m(t,T)(X_t^{(i)})^{2m},\label{S1}
\end{eqnarray}
where $A_m(t,T)$, $m=0,1,2,\dots$, is defined by
\begin{equation}
\label{Am}
A_m(t,T):=\sum_{n=m}^{\infty }%
\int_{[t,T]^{n}}\frac{1}{2^{n}n!}G_{n}^{m}(s_{1},...,s_{n})\left(\prod_{j=1}^n\sigma(s_j)^2\right)\,\mathrm{d}%
s_{n}\ldots\mathrm{d}s_{1}.
\end{equation}

Now we present our main result, which is a significantly simplified expression of \eqref{S1}.
\begin{theorem}
\label{MainThm} Following the definition \eqref{Am}, for any integer $m\ge 2$ and $t\leq T$,
\begin{equation}\label{Main1}
A_0(t,T)^{m-1}A_m(t,T)=\frac{1}{m!}A_1(t,T)^m.
\end{equation}
Therefore, the zero-coupon bond price of the ECIR model defined by \eqref{rel} with $k(t)\equiv0$ is, for any $0\leq t\leq T$,
\begin{equation}\label{Main2}
P(t,T)=A_0(t,T)^de^{-\left(T-t-\frac{A_1(t,T)}{A_0(t,T)}\right)r_t}.
\end{equation}
\end{theorem}
\begin{proof}
Equation \eqref{Main1} is equivalent to 
\begin{equation}\label{Mainequiv}
(m+1)A_0(t,T)A_{m+1}(t,T)=A_1(t,T)A_m(t,T).
\end{equation}
For notation simplicity, we remove time variables $t$ and $T$, and define $B_n^{m}$ for all non-negative integers $m$ and $n$ as
\begin{equation*}
\label{Bmn}
B_n^{m}:= \int_{[t,T]^{n}}\frac{1}{2^{n}n!}G_{n}^{m}(s_{1},...,s_{n})\left(\prod_{j=1}^n\sigma(s_j)^2\right)\,\mathrm{d}%
s_{n}\ldots\mathrm{d}s_{1}.
\end{equation*}
And $C_k$ for all integer $k\ge 0$ as
\begin{equation*}
\label{Cn}
C_k:=\!\!\int_{[t,T]^n}\!\!\!\!\!\!(T-s_{i})(T-s_{i}\vee s_{i_1})...(T-s_{i_k}\vee s_{j})(T-s_{j})\prod_{l=1}^n\sigma(s_{i_l})^2\ud s_{i}\ud s_{i_1}...\ud s_{i_k}\ud s_{j}.
\end{equation*}
By convention, $C_0=\int_{[t,T]^2}(T-s_{i})(T-s_{j})\sigma(s_{i})^2\sigma(s_{j})^2\ud s_{i}\ud s_{j}$ and $C_k=0$ when $k<0$. 
Also denote $C:= \int_{[t,T]}(T-s)^2 \sigma (s)^2 \ud s $. With these notations, it is obvious to see that $A_m = 
\sum_{n=m}^{\infty}B_n^{m}$. Integrating (\ref{recurGnm}) and noting that all integrands are symmetric, we obtain
\begin{equation*}
\label{Bnmrecur}
B_n^{m} = \frac{2}{m}B_{n-1}^{m-1}C +\sum_{k=0}^{n-2}\frac{(-1)^{k+1}2^{k+2}}{mk!}B_{n-k-2}^{m-1}C_k.
\end{equation*}
The left hand side of (\ref{Mainequiv}) can be written as
\begin{eqnarray*}
& &(m+1)A_0A_{m+1}=(m+1)\left( \sum_{p=0}^{\infty} B_p^{0} \right)\left( \sum_{q=m+1}^{\infty} B_{q}^{m+1} \right) \\
&=& \left( \sum_{p=0}^{\infty} B_p^{0} \right)\left( \sum_{q=m+1}^{\infty}\left(2B_{q-1}^{m}C+\sum_{k=0}^{q-2}\frac{(-1)^{k+1}2^{k+2}}{k!}B_{q-k-2}^{m}C_k   \right) \right).
\end{eqnarray*}
For particular $p_0$ and $q_0$, the coefficient of the term $B_{p_0}^0B_{q_0}^m$ is 
\[
2C+\sum_{q=m+1}^{\infty}\frac{(-1)^{q-q_0-1}2^{q-q_{0}}}{(q-q_{0}-2)!}C_{q-q_{0}-2}
\]
which, by the changing of variable $i=q-q_0-2$, is
\[
2C+\sum_{i=m-q_0-1}^{\infty}\frac{(-1)^{i+1}2^{i+2}}{i!}C_{i}.
\]
Similarly, the right hand side of (\ref{Mainequiv}) can be written as
\begin{eqnarray*}
& &A_1A_{m}=\left( \sum_{p=1}^{\infty} B_p^{1} \right)\left( \sum_{q=m}^{\infty} B_{q}^{m} \right) \\
&=&\sum_{p=1}^{\infty} \left(  2B_{p-1}^{0}C+\sum_{k=0}^{p-2}\frac{(-1)^{k+1}2^{k+2}}{k!}B_{p-k-2}^{0}C_k  \right)\left( \sum_{q=m}^{\infty} B_{q}^{m} \right).
\end{eqnarray*}
With the changing of variable $i=p-p_0-2$, the coefficient of $B_{p_0}^0B_{q_0}^m$ is
\[
2C+\sum_{i=-p_0-1}^{\infty}\frac{(-1)^{i+1}2^{i+2}}{i!}C_{i}.
\]
The definition that $G_n^{m}=0$ when $m>n$ implies $q_0\ge m$, which guarantees the left and right hand sides being the same.

To prove \eqref{Main2}, we rewrite \eqref{Main1} as $A_m=A_1^m/(m!A_0^{m-1})$, and substitute it into \eqref{S1} to obtain 
\begin{eqnarray*}
P(t,T)&=&A_0^de^{-(T-t)r_t}\prod_{i=1}^{d}\sum_{m=0}^{\infty }\frac{1}{m!}\left(\frac{A_1}{A_0}(X_t^{(i)})^2\right)^{m}\\
&=&A_0^de^{-(T-t)r_t}\prod_{i=1}^{d}e^{\frac{A_1}{A_0}(X_t^{(i)})^2}=A_0^de^{-\left(T-t-\frac{A_1}{A_0}\right)r_t}.
\end{eqnarray*}\qed 
\end{proof}

Now, we generalize our model to the case of time-dependent drift $k=k(t)$. We use the same notations as
$$F=e^Y,\ \ Y=-\int_t^TX_s^2\ud s,\ \ \ X_s=e^{-\int_0^sk(u)\ud u}\left(x_0+\int_0^se^{\int_0^uk(r)\ud r}\sigma(u)\ud W_u\right).$$
Because of the non-constant drift, we list the two formulas corresponding to \eqref{e1} and \eqref{e2}, which help to generalize Lemma \ref{lemmaGnm}:
\begin{eqnarray}
\label{e11}
\omega^{t}\circ D_{s_{i}}Y&=&-2\left(\int_{s_i}^Te^{-\int_{s_i}^sk(r)\ud r-\int_{t}^sk(r)\ud r}\ud s\right)\sigma(s_i)X_t;\\
\omega^{t}\circ D_{s_{i}s_{j}}Y&=&2\left(\int_{s_{i}\vee s_{j}}^Te^{-\int_{s_i}^sk(r)\ud r-\int_{s_j}^sk(r)\ud r}\ud s\right)\sigma(s_i)\sigma(s_j).\label{e22}
\end{eqnarray}
For simplicity, define
$$\tilde k(a\vee b):=\int_{a\vee b}^Te^{-\int_{a}^sk(r)\ud r-\int_{b}^sk(r)\ud r}\ud s.$$
We generalize the recurrence formulas of $G_n^m$ in the following lemma. 
\begin{lemma}
\label{lemmaGnm2} When $k=k(t)$ is a time dependent nonrandom function, for all non-negative integer $m$, we have  
\begin{equation}
\omega ^{t}\circ \left( D_{s_{n}}^{2}...D_{s_{1}}^{2}F\right) :=(\omega
^{t}\circ F)\sum_{m=0}^{n}G_{n}^{m}(t,s_{1},...,s_{n})\left(\prod_{i=1}^n\sigma(s_i)^2\right)X_t^{2m}
\end{equation}%
where $G_{n}^{m}(t,s_{1},...,s_{n})$ has the following recurrence formula: 

1. $G_0^0=1$ and $G_{n}^{m}=0$ when $m>n$.
 
2. For $n>0$ and $m=0$,
\begin{eqnarray*}
G_{n}^{0}&=&-2\tilde k(s_{n}\vee t)G_{n-1}^{0}(\hat s_{n}) \nonumber \\
&&+\sum_{k=1}^{n-1}\sum_{\{i_{1},i_{2}...i_{k}\}\subset \{1,...,n-1\}}(-1)^{k+1}2^{2k+1}\tilde k(s_{n}\vee s_{i_{1}})\ldots\tilde k(s_{i_k}\vee s_{n})\\
&&\ \ \times G_{n-k-1}^{0}(\hat{s}_{i_{1}},...,\hat{s}%
_{i_{k}}).
\end{eqnarray*}

3. For $n\ge m\ge 1$,
\begin{eqnarray*}
G_{n}^{m}&=&\frac{1}{m}%
\sum_{i=1}^{n}4\tilde k(s_i\vee t)^{2}G_{n-1}^{m-1}(\hat{s}_{i})  \notag \\
&&-\frac{1}{m}\sum_{\{i,j\}\subset \{1,...,n\}}2^5\tilde k(t\vee s_i)\tilde k(s_{i}\vee s_{j})\tilde k(s_{j}\vee t)G_{n-2}^{m-1}(\hat s_i,\hat s_j)\\
&&+\frac{1}{m}\sum_{k=1}^{n-2}\sum_{\substack{ \{i,j\}\subset \{1,...,n\}
\\ \{i_{1},...,i_{k}\}\subset \{1,...,n\}\backslash \{i,j\}}}%
(-1)^{k+1}2^{2k+5}\\
&&\ \ \times\tilde k(t\vee s_i)\tilde k(s_{i}\vee s_{i_{1}})\ldots\tilde k(s_{i_k}\vee s_{j})\tilde k(s_{j}\vee t)\nonumber \\
&&\ \ \times G_{n-k-2}^{m-1}(\hat s_i,\hat s_j,\hat{s}_{i_{1}},...,\hat{s}_{i_{k}}).
\end{eqnarray*}

\end{lemma}
The proof is similar to Lemma \ref{lemmaGnm}, in the sense of replacing \eqref{e1} and \eqref{e2} by \eqref{e11} and \eqref{e22}.

Now, following the same definition of $A_m(t,T)$ in \eqref{Am} with $G_n^m$ defined in Lemma \ref{lemmaGnm2} and the fact that 
$$\omega^t\circ F^{(i)}=e^{\left(-\int_t^Te^{-\int_t^sk(u)\ud u}\ud s\right)(X_t^{(i)})^2},$$ $P(t,T)$ can be calculated similarly to \eqref{S1}, which is 
\begin{eqnarray*}
P(t,T)=e^{-\int_t^Te^{-\int_t^sk(u)\ud u}\ud s}\prod_{i=1}^{d}\sum_{m=0}^{\infty }A_m(t,T)(X_t^{(i)})^{2m}.\label{S2}
\end{eqnarray*}
Our main result Theorem \ref{MainThm} can be directly generalized to the non-constant drift case.
\begin{theorem}
\label{MainThm2} For any $0\leq t\leq T$, the zero-coupon bond price of the ECIR model driven by SDE \eqref{CIR} with $k(t)$ being a time-dependent nonrandom function is,
\begin{equation}\label{Main3}
P(t,T)=A_0(t,T)^de^{-\left(\int_t^Te^{-\int_t^sk(u)\ud u}\ud s-\frac{A_1(t,T)}{A_0(t,T)}\right)r_t}.
\end{equation}
\end{theorem}
The proof follows the same approach as we did for Theorem \ref{MainThm}.

\begin{remark} Our theorems can be applied to solving the time-dependent Riccati equations. From Section 3 of \cite{Mag96}, the bond price is
affine with respect to $r_t$ and satisfies
\begin{equation*}
P(t,T)=A(t,T)e^{-B(t,T)r_t},  \label{generalcase}
\end{equation*}
where $B(t,T)$ solves the time-dependent Riccati equation
\begin{equation}
\frac{\partial B(t,T)}{\partial t}=k(t)B(t,T)+\frac12\sigma (t)^{2}B(t,T)^2-1,  \label{Rica}
\end{equation}%
and $A$ solves $\frac{\partial A(t,T)}{\partial t}=-k(t)\theta(t)B(t,T)$. Theorem \ref{MainThm2} 
allows us to obtain a new solution to the time-dependent Riccati equation (\ref{Rica}) as
\begin{equation*}
B(t,T)=-\frac{A_1(t,T)}{A_0(t,T)}+\int_t^Te^{-\int_t^sk(u)\ud u}\ud s.
\end{equation*}
In the meantime $A(t,T)=A_0(t,T)^d$.
\end{remark}

\section{Numerical Simulations of the Bond Price}
In this section, we give some numerical simulations for the zero-coupon bond price with different volatilities. $\eqref{Main2}$ and \eqref{Main3} can be expanded as
\begin{equation}\label{Main11}
P(t,T)=e^{-\left(T-t\right)r_t}\left(A_0^d+A_0^{d-1}A_1r_t+A_0^{d-1}A_2r_t^2+\ldots\right).
\end{equation}
We apply this expression to numerically calculate the bond price up to a certain order of $T-t$. For the left hand side, we use the Monte-Carlo simulation. For the right hand side, we apply the standard Eular-Maruyama scheme to estimate $r_t$ for fixed $t$ and $T$, and calculate the first few terms of $A_0$ and $A_1$. More specifically, if we assume that $\sigma$ is uniformly bounded by $M>0$, which does not depend on $t$ and $T$, we can write the first few terms of $A_0, A_1$ and $A_2$ as:
\begin{eqnarray*}
A_0&=&1-\int\limits_{t}^{T}(T-s_{1})\sigma(s_{1})^2\, \mathrm{d}
s_{1}\\
&&+\int_{t}^{T}\int_{s_{1}}^{T}\big(2(T-s_{2})^{2}+(T-s_{1})(T-s_{2})\big)%
\sigma(s_{1})^2\sigma(s_{2})^2\, \mathrm{d} s_{2}\, \mathrm{d}
s_{1}\\
&&+O((T-t)^6); \\
A_1&=&\int_t^T 2(T-s_1)^2\sigma(s_1)^2\, \mathrm{d} s_1 \\
&&-\int_t^T \!\!\!\int_{s_1}^T\!\!\!
\big(10(T-s_1)(T-s_2)^2+2(T-s_1)^2(T-s_2)\big)\sigma(s_1)^2\sigma(s_2)^2\, \mathrm{d}
s_2\, \mathrm{d} s_1\\
&&+O((T-t)^7); \\
A_2&=&O((T-t)^6).
\end{eqnarray*} 

In the following examples, we take $k\equiv0, d=1, T=1, t=0.8$ and $r_0=0.5$. 
Based on our selections of $t$ and $T$, we expand our series up to the order $O(T-t)^5\sim 3\times10^{-4}$, i.e., we recombine the first three terms in $A_0$ and the first two terms in $A_1$ in ascending order of $T-t$. We take three volatilities to compare, which are linear decay with $\sigma_1(s)=T-s$ (Figure 1(a)), exponent decay with $\sigma_2(s)=e^{-s}$ (Figure 1(b)), and business cycle with $\sigma_3(s)=\sin s$ (Figure 1(c)), respectively.
\begin{figure}
  \hspace*{-1cm}\includegraphics[width=6in]{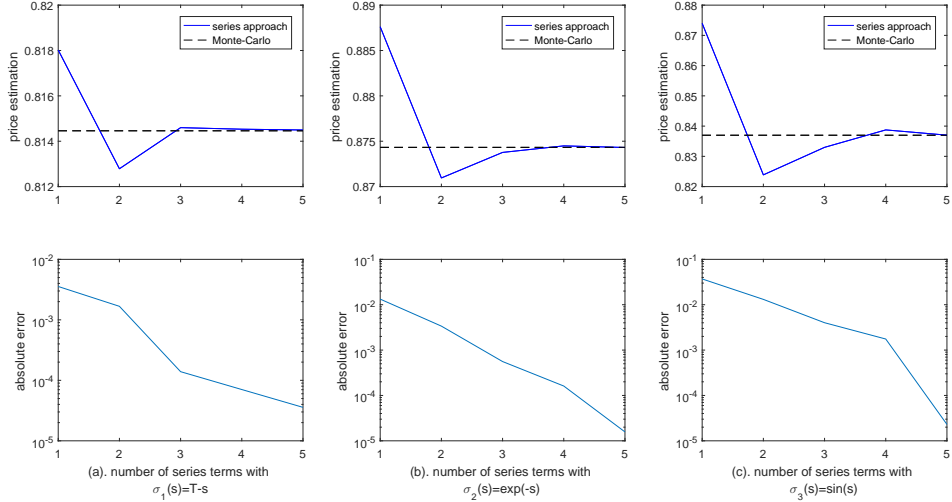}
    \label{F3}
\caption{First five terms of the Dyson type formula.}
\end{figure}
From these figures, it is clear to see that our Dyson type series converges very fast in the first few terms, with the absolute error up to $10^{-5}$, to the true values estimated by the Monte-Carlo simulation. Moreover, the calculation is simple since it only involves Riemann integrals, which can even be calculated manually. On the other hand, our series saves lots of time due to the fact that we only need the current information of the interest rate $r$ at time $t$. In comparison, the Monte-Carlo method requires all future information of $r$ in $(t,T)$ to simulate $-\int_t^Tr_s\ud s$. It is very time consuming to generate a full path of $r_t$ between $t$ and $T$.

\section{Appendix}
\noindent\textit{Proof of Lemma \ref{CONV}}
Based on Lemma \ref{lemmaGnm}, we can bound $G_{n}^{m}$ for all $0\leq m\leq n$. Define $g_{n}^m:=\sup_{s_{1},...,s_{n}\in [t,T]}G_{n}^{m}$ and $\tau=T-t$, then we prove the following bound for $g_{n}^m$: For all $0\leq m\leq n$, there exist two constants $\alpha \geq 1$ and $\beta>2$ which do not depend on $m$, $n$, $t$ and $T$ such that
\begin{equation*}
g_{n}^m\leq \frac{\alpha}{m!}\left(\frac{8}{\beta}\right)^m(4+\beta )^{n}n!\tau^{n+m}.
\end{equation*}

The proof is based on induction on $m$. According to (\ref{recurGn0}), we first build up the recurrence formula for $g_n^0$:%
\begin{equation*}
g_{n}^0\leq\sum_{k=0}^{n-1}2^{2k+1}\tau^{k+1}{\binom{{n-1}}{k}}k!g_{n-k-1}^0.
\end{equation*}%
Now we use induction on $n$ to show that $g_n^0$ satisfies the lemma. First we know $g_0^0=1$. Suppose there exist constants $\alpha$ and $\beta $ such that for all $k\leq
n-1$, 
\begin{equation*}
g_{k}^0\leq \alpha(4+\beta )^{k}k!\tau^{k}.
\end{equation*}%
Then%
\begin{eqnarray*}
g_{n}^0 &\leq &\sum_{k=0}^{n-1}2^{2k+1}\tau^{k+1}{\binom{{n-1}}{k}}%
k!\alpha (4+\beta )^{n-k-1}(n-k-1)!\tau^{n-k-1} \\
&=&2\alpha (n-1)!(4+\beta )^{n-1}\tau^{n}\sum_{k=0}^{n-1}\frac{4^{k}}{%
(4+\beta )^{k}} \\
&\leq &\alpha (4+\beta )^{n}n!\tau^{n},
\end{eqnarray*}%
which proves the case for $m=0$.
Now suppose that $g_n^m$ satisfies the lemma for some $1\leq m \leq n-1$, then based on (\ref{recurGnm}), we build up the recurrence formula for $g_n^{m+1}$:
\begin{equation*}
g_n^{m+1}\leq\frac{1}{m+1}\sum_{k=1}^n \frac{n!2^{2k+1}}{(n-k)!}\tau^{k+1}g_{n-k}^m.
\end{equation*}
By induction assumption,
\begin{eqnarray*}
g_n^{m+1}&\leq &\frac{1}{m+1}\sum_{k=1}^n\frac{2n!4^k}{(n-k)!}\tau^{k+1}\frac{\alpha}{m!}\left(\frac{8}{\beta}\right)^m(4+\beta)^{n-k} (n-k)!\tau^{n+m-k} \\
&=&\frac{\alpha n!(4+\beta)^n}{(m+1)!}\left(\frac{8}{\beta}\right)^m\tau^{n+m+1}\sum_{k=1}^n\frac{2\cdot4^k}{(4+\beta)^k} \\
&\leq&\frac{\alpha n!(4+\beta)^n}{(m+1)!}\left(\frac{8}{\beta}\right)^{m+1}\tau^{n+m+1}.
\end{eqnarray*}
Thus by (\ref{Gnmdef}), the Cauchy-Schwarz inequality, Stirling approximation $n!\sim \sqrt{2\pi}n^{n+\frac{1}{2}}e^{-n}$ and the moment generating function of the geometric Brownian motion $E\left[e^{zW_t}\right]=e^{\frac{z^2}{2}t}$, we obtain,
\begin{eqnarray*}
&&\frac{(T-t)^{2n}}{(2^{n}n!)^{2}}E\left[ \left( \sup_{u_{1},...,u_{n}\in
\lbrack t,T]}\left\vert \omega^t \circ (D_{s_{n}}^{2}...D_{s_{1}}^{2}F)\right\vert \right) ^{2}\right] \\
&\leq&\frac{(T-t)^{2n}n}{(2^{n}n!)^{2}}E[\omega^t \circ F]^2 \left(\sum_{m=1}^n(g_n^m)^2E[W_t^{4m}]\right)\\
&\leq&\frac{(T-t)^{2n}n}{(2^{n}n!)^{2}}E[\omega^t \circ F]^2 \left(\sum_{m=1}^n \frac{\alpha^2 (n!)^2(4+\beta)^{2n}}{(m!)^2}\left(\frac{8}{\beta}\right)^{2m}\tau^{2n+2m}\frac{(4m)!t^{2m}}{(2m)!2^{2m}}\right)\\
&\leq&C\left(\frac{\tau^2(4+\beta)}{2}\right)^{2n}n\sum_{m=1}^n\left(\frac{8t\tau}{\beta}\right)^{2m}
\end{eqnarray*}
where $C$ is a constant which does not depend on $m$ and $n$.
Thus if we fix $t$ and $T$, take $\beta$ being large enough, then the condition (\ref{conv}) holds.
\qed 

\bigskip
\noindent\textit{Proof of Lemma \ref{Lemma2}}
When $n=1$, $\alpha$ can be either $\{\{1\},\{1\}\}$ or $\{\{1,1\}\}$,
which corresponds to $\left(D_{s_{1}}Y\right)^{2}$ or $D_{s_{1}}^2Y$
respectively.  Now we prove the lemma by induction. Assuming true for the case $n$, we calculate the $n+1$ case as
\begin{eqnarray}
\label{LL1}
&&D_{s_{n+1}}^{2}D_{s_{n}}^{2}\ldots D_{s_{1}}^{2}F\notag \\
&=&\sum_{\alpha\in{\cal B}_{n}}D_{s_{n+1}}^{2}\left(FK_{\alpha}\right)\notag \\
&=&\sum_{\alpha\in{\cal B}_{n}}F\left(\left(D_{s_{n+1}}Y\right)^{2}K_{\alpha}
+D_{s_{n+1}}^2YK_{\alpha}+2D_{s_{n+1}}YD_{s_{n+1}}K_{\alpha}
+D_{s_{n+1}}^2K_{\alpha}\right).\notag\\
\end{eqnarray}

For arbitrary $\alpha\in \mathcal{B}_n$, the first term $\left(D_{s_{n+1}}Y\right)^{2}K_{\alpha}$ corresponds
to the partition $\alpha\cup\{\{n+1\},\{n+1\}\}\in\mathcal{B}_{n+1}$, and the second term
$D_{s_{n+1}}^2YK_{\alpha}$ corresponds to the partition $\alpha\cup\{\{n+1,n+1\}\}\in\mathcal{B}_{n+1}$. 
Now we consider the term $2D_{s_{n+1}}YD_{s_{n+1}}K_{\alpha}$. Since
taking Malliavin derivatives on $Y$ more than twice is $0$,
we expand $D_{s_{n+1}}K_{\alpha}$ using chain rule and $D_{s_{n+1}}$
can only act on the term $D_{s_{i}}Y$, which corresponds
to the subset of Case 1 in Table 1. If $\{\{i\},\{i\}\}\subset\alpha$, $2D_{s_{n+1}}YD_{s_{n+1}}K_{\alpha}$ corresponds to the partition $\alpha\cup\{\{n+1\},\{n+1,i\}\}\backslash\{\{i\}\}\in\mathcal{B}_{n+1}$.
Otherwise, there must exist a subset of Case $5$ contains $i$, namely
$\{i,i_{1},i_{1},\ldots,i_{k-1},i_{k-1},i_{k}\}\in\alpha$. In this case, $2D_{s_{n+1}}YD_{s_{n+1}}K_{\alpha}$ corresponds to the partition 
$$\alpha\cup \big\{\{n+1\},\{n+1,i,i,i_{1},i_{1},...,i_{k-1},i_{k-1},i_{k}\}\big\}\backslash\big\{\{i\},\{i,i_{1},i_{1},...,i_{k-1},i_{k-1},i_{k}\}\big\}$$
in $\mathcal{B}_{n+1}$.
Finally we consider the term $D_{s_{n+1}}^2K_{\alpha}$. Similarly, $D_{s_{n+1}}$ must act separately on two terms $D_{s_i}Y$ and 
$D_{s_j}Y$ in $K_\alpha$, which come from Case 1 again. Thus, $D_{s_{n+1}}^2K_{\alpha}$ corresponds to either the partition $\alpha\cup\{\{n+1,n+1,i,i\}\}\in\mathcal{B}_{n+1}$ for $i=j$, or the partition $$\alpha\cup\{\{n+1, n+1,i,i,i_1,i_1...,i_k,i_k,j,j\}\}\backslash\{\{i\},\{i,i_1,i_1,...,i_k,i_k,j\},\{j\}\}$$ 
in $\mathcal{B}_{n+1}$ for $i\ne j$.

On the other hand, for arbitrary $\alpha^{*}\in{\cal B}_{n+1}$, we need to show that it comes from one of the terms on the right hand side of \eqref{LL1}. Essentially, this is converse to the approach used in the previous paragraph. We choose the subsets that contain the element $n+1$ and correspond them to the right hand side of \eqref{LL1}. For example, if $\{\{n+1\},\{n+1\}\}\in\alpha^*$, then it corresponds to $D_{s_{n+1}}^2YK_{\alpha}$, with $\alpha=\alpha^*\backslash\{\{n+1\},\{n+1\}\}\in\mathcal{B}_n$. It is not hard to check other cases similarly. Therefore, we proved that the right hand side of \eqref{LL1} equals to $\sum_{\alpha^*\in\mathcal{B}_{n+1}}FK_{\alpha^*}$, which yields the lemma.\qed

\bigskip
\noindent\textit{Proof of Lemma \ref{lemmaGnm}} 
We first list two expressions which will be repeatedly used:
\begin{eqnarray}
\label{e1}
\omega^{t}\circ D_{s_{i}}Y&=&-2(T-s_i)\sigma(s_i)X_t;\\
\omega^{t}\circ D_{s_{i}s_{j}}Y&=&-2(T-s_{i}\vee s_{j})\sigma(s_i)\sigma(s_j).\label{e2}
\end{eqnarray}
Since $G_{n}^{0}$ denotes the terms with $m=0$ in \eqref{Gnmdef} where $X_t$ is not included, by \eqref{e1}, first order derivative of $Y$ does not contribute in $G_{n}^{0}$.
Thus, to prove (\ref{recurGn0}), we only need to consider the set partition $\alpha$ which contains
the subsets of Cases $2$, $3$ and $4$ from Table 1. We classify $\alpha$
by the cases of subsets containing element $n$.

i). If $\{n,n\}\in\alpha:$
\begin{eqnarray*}
\sum_{\alpha\in{\cal B}_{n}, \ \{n,n\}\in\alpha}\omega^{0}\circ K_{\alpha}&=&\left(\omega^{0}\circ D_{s_{n}s_{n}}Y\right)\left(\sum_{\alpha\in{\cal B}_{n-1}}\omega^{0}\circ K_{\alpha}\right)\\
&=&-2(T-s_{n})\sigma(s_n)^2G_{n-1}^{0}(s_{1},s_{2},...,s_{n-1}).
\end{eqnarray*}

ii). If $\{n,n,i,i\}\in\alpha:$
\begin{eqnarray*}
\sum_{\alpha\in{\cal B}_{n},\{n,n,i,i\}\in\alpha}\omega^{0}\circ K_{\alpha}&=&\left(\omega^{0}\circ2\left(D_{s_{n}s_{i}}Y\right)^{2}\right)\left(\sum_{\alpha\in{\cal B}_{n-2}}\omega^{0}\circ K_{\alpha}\right)\\
&=&2^{3}(T-s_{n}\vee s_{i})^{2}\sigma(s_n)^2\sigma(s_i)^2G_{n-2}^{0}(\hat{s}_{i}).
\end{eqnarray*}

iii) If $\{n,n,i_{1},i_{1},i_{2},i_{2},\ldots,i_{k},i_{k}\}\in\alpha$ with $k\geq2$:
\begin{eqnarray*}
&&\sum_{\alpha\in{\cal B}_{n},\{n,n,i_{1},i_{1},i_{2},i_{2},...,i_{k},i_{k}\}\in\alpha}\omega^{0}\circ K_{\alpha}\\
&=&\left(\omega^{0}\circ2^{k+1}\left(D_{s_{n}s_{i_{1}}}Y\right)\left(D_{s_{i_{1}}s_{i_{2}}}
Y\right)\ldots\left(D_{s_{i_{k}}s_{n}}Y\right)\right)\left(\sum_{\alpha\in{\cal B}_{n-k-1}}\omega^{0}\circ K_{\alpha}\right)\\
&=&(-1)^{k+1}2^{2k+1}(T-s_{n}\vee s_{i_{1}})(T-s_{i_{1}}\vee s_{i_{2}})\ldots(T-s_{i_{k}}\vee s_{n})\\
&&\ \ \ \times \sigma(s_n)^2\prod_{j=1}^k\sigma(s_{i_j})^2G_{n-k-1}^{0}(\hat{s}_{i_{1}},\hat{s}_{i_{2}},...,\hat{s}_{i_{k}},\hat s_n).
\end{eqnarray*}
Combining these three cases, we proved (\ref{recurGn0}).

Now we prove (\ref{recurGnm}). By \eqref{e1} and \eqref{e2}, the terms containing $X_{t}^{2m}$ come from the set partition $\alpha$, in which there are $2m$ subsets of Case $1$ of Table 1. Let ${\cal B}_{n}(m)$ be a subset
of ${\cal B}_{n}$ defined by
\[
{\cal B}_{n}(m):=\left\{ \alpha:\alpha\in{\cal B}_{n},\ K_{\alpha}\ \text{contains }X_{t}^{2m}\right\} .
\]
Then we have
\[
\omega^{t}\circ\left(D_{s_{n}}^{2}\ldots D_{s_{1}}^{2}F\right)=\left(\omega^{t}\circ F\right)\sum_{m=0}^{n}\sum_{\alpha\in{\cal B}_{n}(m)}\omega^{t}\circ K_{\alpha}.
\]
Assume $\{i\}\in\alpha$. We classify $\alpha\in{\cal B}_{n}(m)$ by
the cases of subsets containing the other element $i$.

i) If $\{\{i\},\{i\}\}\subset\alpha$:
\begin{eqnarray*}
\sum_{\substack{\alpha\in{\cal B}_{n}(m),\\ \{\{i\},\{i\}\}\subset\alpha}}\omega^{t}\circ K_{\alpha}&=&\left(\omega^{t}\circ\left(D_{s_{i}}Y\right)^{2}\right)\left(\sum_{\alpha\in{\cal B}_{n-1}(m-1)}\omega^{t}\circ K_{\alpha}\right)\\
&=&4(T-s_{i})^{2}\sigma(s_i)^2G_{n-1}^{m-1}(\hat{s}_{i})X_{t}^{2}.
\end{eqnarray*}

ii) If $\{\{i\},\{i,i_{1},i_{1},i_{2},i_{2},...,i_{k},i_{k},j\},\{j\}\}\subset\alpha$ for $k\geq0$:
\begin{eqnarray*}
&&\sum_{\substack{\alpha\in{\cal B}_{n}(m),\\ \{\{i\},\{i,i_{1},i_{1},...,i_{k},i_{k},j\},\{j\}\}\subset\alpha}}\omega^{t}\circ K_{\alpha}\\
&=&\left(\omega^{t}\circ2^{k+2}D_{s_{i}}YD_{s_{i}s_{i_{1}}}YD_{s_{i_{1}}s_{i_{2}}}Y... D_{s_{i_{k}}s_{j}}YD_{s_{j}}Y\right)\left(\sum_{\alpha\in{\cal B}_{n-k-2}(m-1)}\!\!\!\omega^{t}\circ K_{\alpha}\right) \\
&=&(-1)^{k+1}2^{2k+5}(T-s_{i})(T-s_{i}\vee s_{i_{1}})\ldots(T-s_{i_{k}}\vee s_{j})(T-s_{j})\\
&&\ \ \ \times \sigma(s_i)^2\sigma(s_j)^2\prod_{u=1}^k\sigma(s_{i_u})^2G_{n-k-2}^{m-1}(\hat{s}_i,\hat{s}_j,\hat{s}_{i_{1}},...,\hat{s}_{i_{k}})X_{t}^{2}.
\end{eqnarray*}
Combining these two cases, we proved (\ref{recurGnm}).\qed 

\textbf{Acknowledgements:} The authors would like to thank Professor Henry Schellhorn at
Claremont Graduate University for helpful discussions.






%
%



%
%
\end{document}